\documentclass{article}
\pdfoutput=1
\usepackage[utf8]{inputenc}
\usepackage[T1]{fontenc}
\usepackage{amsmath,amssymb,amsfonts,amsthm,graphicx,mathrsfs,url}
\usepackage[usenames,dvipsnames]{color}
\usepackage[colorlinks=true,linkcolor=Red,citecolor=Green,urlcolor=Blue]{hyperref}
\usepackage[super]{nth}
\usepackage[open, openlevel=2, depth=3, atend]{bookmark}
\hypersetup{pdfstartview=XYZ}

\usepackage{bbold}

\def\?[#1]{\textbf{[#1]}\marginpar{\Large{\textbf{??}}}}

\numberwithin{equation}{section}

\newtheorem{rem}{Remark}

\newtheorem{definition}{Definition}

\DeclareMathSymbol{\leqslant}{\mathalpha}{AMSa}{"36} % nicer `smaller or equal'
\DeclareMathSymbol{\geqslant}{\mathalpha}{AMSa}{"3E} % nicer `larger or equal'
\DeclareMathSymbol{\eset}{\mathalpha}{AMSb}{"3F}     % nicer `emptyset'
\renewcommand{\leq}{\;\leqslant\;}                   % redef. of < or =
\renewcommand{\geq}{\;\geqslant\;}                   % redef. of > or =
\def\bi{\begin{itemize}}
\def\ei{\end{itemize}}

\def\bnum{\begin{enumerate}}
\def\enum{\end{enumerate}}
\def\<#1{\langle #1 \rangle}

\usepackage{cite}

\def\g{\mathbf{g}}

\newtheorem{thm}{Theorem}
\newtheorem{bij}{Bijection}
\newtheorem*{spli}{Splitting operation}

\newtheorem{cor}{Corollary}
\newtheorem{lem}{Lemma}
\usepackage{caption}
\usepackage[capitalise]{cleveref}

\Crefname{lem}{Lemma}{Lemmas}

\Crefname{thm}{Theorem}{Theorems}

\Crefname{cor}{Corollary}{Corollaries}
\Crefname{bij}{Bijection}{Bijections}

\Crefname{rem}{Remark}{Remarks}
\crefname{rem}{remark}{remarks}

\title{A new family of bijections for planar maps}

\author{Baptiste LOUF \\ IRIF, Université Paris Diderot - Paris 7 \\
Bâtiment Sophie Germain, 75205 Paris Cedex 13, France \\
blouf@irif.fr}

\begin{document}

%
 %\vspace{1cm}
%\begin{abstract}
\maketitle
\abstract{
We present bijections for the planar cases of two counting formulas on maps that arise from the KP hierarchy (Goulden-Jackson and Carrell-Chapuy formulas), relying on a "cut-and-slide" operation. 
This is the first time a bijective proof is given for quadratic map-counting formulas derived from the KP hierarchy. Up to now, only the linear one-faced case was known (Harer-Zagier recurrence and Chapuy-F\'eray-Fusy bijection).
As far as we know, this bijection is new and not equivalent to any of the well-known bijections between planar maps and tree-like objects.}
%\end{abstract}

\vspace{0.3cm}
\footnotesize


%\noindent{\bf MSC 2000 subject classifications:  60D05, 81T40,  81T20.}    

\normalsize

\section{Introduction}

\textbf{Context.} 
A map is  a combinatorial object describing the embedding up to homeomorphism of a multigraph on a compact orientable surface (see  Section \ref{sec:def} for precise definitions).
Map enumeration has been an important research topic for many years now. Tutte first enumerated planar maps and found simple formulas \cite{Tutte}, raising the natural concern of finding bijections explaining those formulas. Such bijections have since been found, mostly thanks to the family of bijections between maps and decorated trees (blossoming trees and mobiles) \cite{Cori1981,theseSch,BDG,AP,BF}. 
Some of these bijections have recently been extended to maps on surfaces of higher genus (see for instance \cite{Schaeffer,Lep}) and on non-orientable surfaces \cite{CD,BettiNO}. These bijections also opened the way to the study of random maps and their scaling limits (see e.g. \cite{CS,Miermont,LG}).

Another powerful tool for map enumeration is the KP hierarchy. The Kadomtsev-Petviashvili hierarchy is an infinite set of PDE's on functions with an infinite number of variables $(p_i)_{i\geq 1}$, which arose from mathematical physics (see \cite{Solitons},\cite{Ok} and references therein). The first equation of the hierarchy is
\[{F_{3,1}=\frac{1}{12}F_{1,1,1,1}+F_{2,2}+\frac{1}{2}(F_{1,1})^2},\]
where the indices indicate partial derivatives (e.g. $F_{3,1}=\frac{\partial^2}{\partial p_1\partial p_3}F$). 
%First, links with objects connected to maps, such as Hurwitz numbers, have been observed (see for instance \cite{Ok}). 
The link between integrable hierarchies and maps was observed in mathematical physics through matrix integrals (see e.g. \cite{LZ} and references therein).
In 2008, Goulden and Jackson \cite{GJ} showed that certain generating functions for maps (with $p_i$ counting vertices of degree $i$) are solutions to the KP hierarchy. It allowed them to derive a very simple recurrence formula
\begin{equation}
\begin{split}
(n+1)T(n,g)=&4n(3n-2)(3n-4)T(n-2,g-1)+4(3n-1)T(n-1,g)\\
&+4\sum_{\substack{i+j=n-2\\i,j\geq 0}}\sum_{\substack{g_1+g_2=g\\g_1,g_2\geq 0}}  (3i+2)(3j+2)T(i,g_1)T(j,g_2)+2\mathbb{1}_{n=g=1},
\end{split}
\label{GJ}
\end{equation}
where $T(n,g)$ is the number of rooted cubic maps of genus $g$ with $3n$ edges (equivalently, triangulations — with loops and multiple edges — of genus $g$ with $3n$ edges). Then, using similar methods, Carrell and Chapuy \cite{CC} proved a recurrence formula on general (i.e. not necessarily cubic nor triangular) maps
\begin{equation}
\begin{split}
(n+1)Q_g(n,f)=&2(2n-1)Q_g(n-1,f)+2(2n-1)Q_g(n-1,f-1)\\
&+(2n-3)(n-1)(2n-1)Q_{g-1}(n-2,f)\\
&+3\sum_{\substack{i+j=n-2\\i,j\geq 0}}\sum_{\substack{f_1+f_2=f\\f_1,f_2\geq 1}}\sum_{\substack{g_1+g_2=g\\g_1,g_2\geq 0}}(2i+1)(2j+1)Q_{g_1}(i,f_1)Q_{g_2}(j,f_2),
\end{split}
\label{CC}
\end{equation}
where $Q_g(n,f)$ is the number of rooted maps of genus $g$ with $n$ edges and $f$ faces. Taking $f=1$ in \eqref{CC}, one recovers the famous Harer-Zagier recurrence \cite{HZ}.

Finding bijections for formulas arising from the KP hierarchy on maps would allow us to understand maps in greater depth, but for now it is still mainly an open problem. The only special case of \eqref{GJ} and \eqref{CC} for which a bijective proof is known is the case of one-faced maps \cite{CFF}.

\textbf{Contributions of this article.} 
In this paper, we present bijective proofs for the planar case ($g=0$) of Goulden--Jackson and Carrell--Chapuy formulas \eqref{GJ} and \eqref{CC}. Note that contrarily to the one-faced case which is linear, the planar formulas are, as in the general case, quadratic.

We actually prove a more general formula on precubic maps (see Section \ref{sec:def} for a definition) which implies the Goulden--Jackson formula \eqref{GJ}. The Carrell--Chapuy formula \eqref{CC} comes from two separate (but somehow related by their bijective proofs) formulas that were not predicted by the KP hierarchy, one of them being a generalization of the famous R\'emy bijection on plane trees \cite{Remy} to all planar maps (see formulas \eqref{cut-slide} and \eqref{remy}).

Our bijections rely on a particular exploration of the map and on a "cut-and-slide" operation. Although non-local, this operation allows us to keep track of the degrees of the vertices (see Section \ref{sec:deg}, Theorem \ref{thm_degrees}). We also take a first step towards the higher genus case by proving the Goulden-Jackson formula for two-faced maps in any genus (see Section \ref{sec:twofaces} for a sketched proof).

\textbf{Comments and related works.} 
The bijective study of planar maps is a well understood topic, especially thanks to bijections with tree-like objects. However, as far as we know, our bijection is not equivalent to those bijections (although certain similarities can be observed, such as a search of the dual map, DFS in our case, BFS in the bijections above). A related cut-and-slide operation appeared in \cite{Bettinelli}, but it is different from ours (see Remark \ref{rem} for more details).

It is a natural question to try to unify bijections in this paper (that apply to planar maps with arbitrary number of faces) and bijections in \cite{CFF} (that apply to unicellular maps with arbitrary genus) in order to prove \eqref{GJ} and \eqref{CC} in full generality. These two approaches seem different at first sight, but they both consist of taking a map with a marked special half-edge, and "cutting" at this special half-edge in order to obtain one or two maps with less edges, faces, or lower genus. In our case, those special half-edges are called \textit{discoveries}, and there are $f-1$ of them in a map with $f$ faces ; in \cite{CFF}, they are called \textit{trisections}, and there are $2g$ of them in a map of genus $g$. In \eqref{GJ}, by the Euler formula, the factor $(n+1)$ in the RHS is equal to $2g+f-1$ if there are $f$ faces, this suggests that in the general case these special half-edges still exist and there would be $2g+f-1$ of them. They would thus be a common generalization of trisections and discoveries. We already found a bijective proof of the precubic recurrence formula for two faced maps (a sketch of proof is included in Section \ref{sec:twofaces}), but it is already quite complicated and involves several cases. That is why putting it all together seems to be a challenge of its own. In another direction, we can also look at the second KP equation (which is also quadratic), which yields a recurrence formula (very similar to \eqref{precubic}) for planar maps with vertices of degrees $1$ or $4$, and this formula is also proven by our bijection (it is a specialization of the more general formula obtained in Section \ref{sec:deg}). This suggests that this exploration+cut-and-slide scheme is maybe somehow underlying in the whole KP hierarchy for maps. 

\textbf{Structure of the paper.} 
In Section \ref{sec:def}, we will give some definitions on maps and state the main results. The bijections will be described in Section \ref{sec:bij} (including the proof of the Goulden--Jackson formula), with some proofs postponed to Section \ref{sec:proof}. Section \ref{sec:deg} will present refined formulas with control over the degrees of the vertices. In Section \ref{sec:calcul}, we prove the Carrell-Chapuy formula using formulas \eqref{cut-slide} and \eqref{remy}. In Section \ref{sec:twofaces}, we give a sketch of proof of the precubic recurrences for two-faced maps in any genus.

\section{Definitions and main results}
\label{sec:def}
\label{sec:results}
\begin{definition}
A map $M$ is the data of a connected multigraph (multiple edges and loops are allowed) $G$ (called the underlying graph) embedded in a compact connected oriented surface $S$, such that $S\setminus G$ is homeomorphic to a collection of disks. The connected components of $S\setminus G$ are called the \textit{faces}. $M$ is defined up to homeomorphism. Equivalently, $M$ is the data of $G$ and a \textbf{rotation system} which describes the clockwise cyclic order of the half-edges around each vertex.
The \textbf{genus} $g$ of $M$ is the genus of $S$ (the number of "handles" in $S$). A \textbf{corner} of $M$ is an angular sector between two consecutive half-edges around a vertex. A \textbf{rooted map} is a map with a distinguished corner. A small arrow is placed in the distinguished corner, thus splitting the corner in two separate corners (left and right of the arrow).
If a rooted map $M$ of genus $g$ has $n$ edges, $v$ vertices and $f$ faces, the Euler formula links those quantities: $v-n+f=2-2g$. $M$ has $2n+1$ corners.

A \textbf{planar map} is a rooted map of genus $0$. It can be drawn on the plane with the root lying on the outer face. A \textbf{precubic map} is a map with vertices of degree $1$ or $3$ only, rooted on a vertex of degree $1$. A \textbf{leaf} is a vertex of degree $1$ that is not the root.
\end{definition}

From now on, all the maps we consider will be rooted, and planar unless stated otherwise.

\begin{figure}%[!h]
\captionsetup{width=0.8\textwidth}
\centering
\includegraphics[scale=0.5]{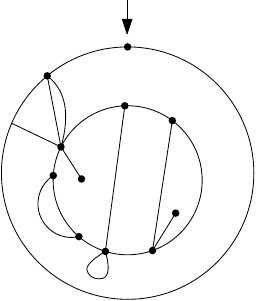}
\caption{A planar map}
\label{map}
\end{figure}

In Section \ref{sec:bij}, we will introduce the concept of discovery, which are special edges obtained by an exploration of the map. An important feature is that a map with $f$ faces has $f-1$ discoveries.

\begin{thm}
\label{cs}
There is a bijection between planar maps $M$ with a marked discovery and pairs of planar maps $(M_1,M_2)$ such that $M_1$ has a marked vertex and $M_2$ has a marked leaf. This bijection preserves the total number of edges and faces. This gives the following formula on planar maps
\begin{equation}
(f-1)Q(n,f)=\sum_{\substack{i+j=n-1\\i,j\geq 0}}\sum_{\substack{f_1+f_2=f\\f_1,f_2\geq 1}}v_1Q(i,f_1)(2j+1)Q(j,f_2),
\label{cut-slide}
\end{equation}
where $Q(n,f)$ is the number of planar maps with $n$ edges and $f$ faces, and $v_1$ counts the number of vertices in the first map (i.e. $v_1=2+i-f_1$).

This bijection adapts to precubic maps, and the marked vertex of $M_1$ is now a marked leaf (see Remark \ref{rem_precubic}). This yields
\begin{equation}
(f-1)\alpha(n,f)=\sum\limits_{i+j=n}\sum\limits_{f_1+f_2=f}\alpha^{(1)}(i,f_1)\alpha^{(1)}(j,f_2),
\label{precubic}
\end{equation}
where $\alpha(n,f)$ counts the number of (planar) precubic maps with $n$ edges and $f$ faces, and $\alpha^{(1)}(n,f)$ counts the number of precubic maps with $n$ edges and $f$ faces and a marked leaf. 
\end{thm}

Retracting a leaf into a corner (see Figure \ref{feuille} left) is a classical operation that gives a bijection between planar maps $M$ with a marked leaf and planar maps $M'$ with a marked corner, such that $M'$ has as many faces and one edge less than $M$. This operation, along with a more subtle retraction operation on internal nodes, is at the core of the Rémy bijection on plane trees \cite{Remy}. We generalize the operation on internal nodes to all planar maps in the following:
\begin{thm}[Generalized R\'emy bijection]
\label{re}

There is a bijection between planar maps $M$ with a marked node (i.e. a vertex that is not a leaf) on the one hand, and the union of planar maps $M'$ with a marked corner and pairs of planar maps $(M_1,M_2)$ such that they both have a marked vertex on the other hand. The total number of faces is preserved, and the total number of edges decreases by one.

By this bijection, there are \[(2n-1)Q(n-1,f)+\sum_{\substack{i+j=n-1\\i,j\geq 0}}\sum_{\substack{f_1+f_2=f\\f_1,f_2\geq 1}}v_1Q(i,f_1)v_2Q(j,f_2)\] planar maps with $n$ edges, $f$ faces and a marked node. The retraction operation (of Figure \ref{feuille}) implies that there are  $(2n-1)Q(n-1,f)$ planar maps with $n$ edges, $f$ faces and a marked leaf. This yields the following equation:
\begin{equation}
vQ(n,f)=2(2n-1)Q(n-1,f)+\sum_{\substack{i+j=n-1\\i,j\geq 0}}\sum_{\substack{f_1+f_2=f\\f_1,f_2\geq 1}}v_1Q(i,f_1)v_2Q(j,f_2),
\label{remy}
\end{equation}
where the "$v$-variables" count the number of vertices (i.e. $v=2+n-f$, $v_1=2+i-f_1$ and $v_2=2+j-f_2$). This holds for $n>0$.
\end{thm}
Taking $n=3m+2$ and $f=m+2$ in \eqref{precubic}, one recovers
\begin{cor}[Goulden-Jackson \eqref{GJ} planar case]
\label{cor}
\[(n+1)T(n)=4(3n-1)T(n-1)+4\sum_{\substack{i+j=n-2\\i,j\geq 0}}  (3i+2)(3j+2)T(i)T(j),\]
where $T(n)$ counts the number of planar cubic maps with $3n$ edges.
\end{cor}

\begin{rem}
The term $4(3n-1)T(n-1)$ corresponds to $i=f_1=1$ and $j=f_2=1$ in the summation of \eqref{precubic} (because $\alpha^{(1)}(1,1)=1$).
\end{rem}

Combining  \eqref{cut-slide} and \eqref{remy} and doing some manipulations, one recovers
\begin{cor}[Carrell-Chapuy \eqref{CC} planar case]
\begin{equation}
\label{planar}
\begin{split}
\begin{split}
(n+1)Q(n,f)=&2(2n-1)Q(n-1,f)+2(2n-1)Q(n-1,f-1)\\
&+3\sum_{\substack{i+j=n-2\\i,j\geq 0}}\sum_{\substack{f_1+f_2=f\\f_1,f_2\geq 1}}(2i+1)(2j+1)Q(i,f_1)Q(j,f_2).
\end{split}
\end{split}
\end{equation}
\end{cor}

In Section \ref{sec:bij}, we will give the definition of the exploration and the discoveries, justify \eqref{cut-slide} and \eqref{precubic} and Corollary \ref{cor}, then we will describe the bijections behind Theorems \ref{cs} and \ref{re}. The proof that these are indeed bijections can be found in Section \ref{sec:proof}.
Formula \eqref{planar} is not straightforwardly derived from \eqref{cut-slide} and \eqref{remy}. The calculations to recover \eqref{planar} from there are displayed in Section \ref{sec:calcul}.
\section{The bijections}
\label{sec:bij}
In this section, we will define the exploration of a planar map and the notion of discoveries that result from it, then we will explain our bijections.

\subsection{The exploration}
\label{explo}
\begin{definition}
The \textbf{exploration} of a planar map (see Figure \ref{exploration}) is defined iteratively in the following way: starting from the root, go along the edges, keeping the edges on the right (progress in clockwise order). When an edge that is at the interface of the current face and a face not yet discovered is found, open this edge into a bud (an outgoing half-edge) and a stem (an ingoing half-edge). The rule is that the bud has to appear before the stem in the exploration (see Figure \ref{exploration}). Continue the process, thus entering the new face. Continue until the root is reached again.

Each edge that has been opened during the process is called a \textbf{discovery}, and the vertex attached to the bud is called a \textbf{discovery vertex}. If there are $f$ faces , there are $f-1$ discoveries (note that several discoveries can share the same discovery vertex).

The exploration is actually equivalent to a DFS of the dual, with a "right first" priority. Thus for each face but the outer face we can define its \textbf{previous face} as the face that is its parent in the spanning tree of the dual found by the exploration. The notion of \textbf{previous discovery} can be similarly defined. This also defines an order on the corners (resp. half-edges) incident to each vertex, according to the order in which they were visited during the exploration (see Figure \ref{exploration}). 

Let $e$ be a discovery, incident to faces $f_1$ and $f_2$, such that $f_1$ is the previous face of $f_2$. We say $e$ \textbf{leaves} $f_1$ and \textbf{enters} $f_2$.
\end{definition}

\begin{rem}
The exploration is a dynamic process that modifies the map along the way, but in the end, once the exploration is over and the discoveries have been found, we will deal with the original, unmodified map, with its original edges and faces. It is as if we did the exploration then closed the map back. Alternatively, one can think of an exploration that does not open the discoveries but just crosses them.
\end{rem}
\begin{figure}[!h]
\captionsetup{width=0.8\textwidth}

\centering

\includegraphics[scale=0.7]{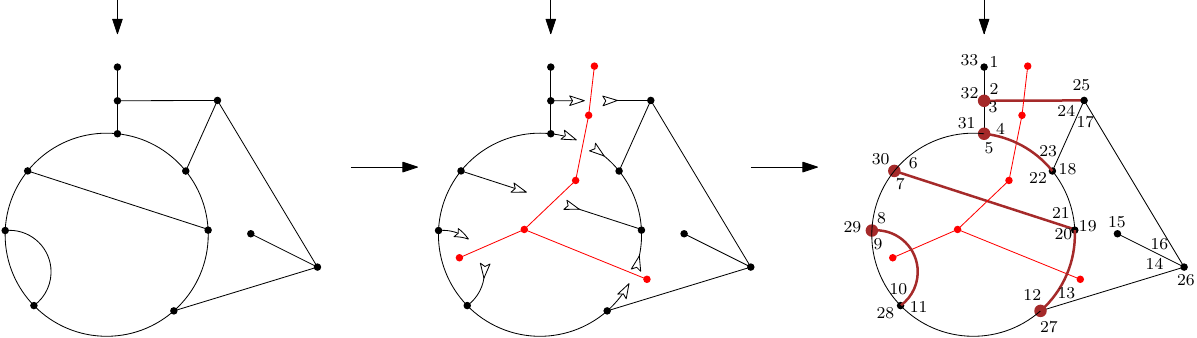}
\caption{The exploration of a planar map. The buds are the outgoing arrows, the stems are the ingoing arrows. Left: the original map. Center: the opened map. Right: The original map, with its discoveries and discovery vertices in fat brown. The red tree describes the partial order among the faces. The corners are labeled in the order they were found during the discovery}
\label{exploration}
\end{figure}

\begin{lem}
\label{clockwise}
Around each vertex, the order of the corners as defined by the exploration agrees with the clockwise order.
\end{lem}

\begin{proof}
The order between the corners of the map is exactly the same as the order of the corners of the blossoming tree (a blossoming tree is a tree with some buds and stems attached) obtained by opening the discoveries. The exploration of a tree is just a tour of the unique face, and it is clear that the corners of each vertex are in clockwise order, and we are done.
\end{proof}

We can now relate the bijections and the formulas. In a map with $f$ faces, there are $f-1$ discoveries, so there are $(f-1)Q(n,f)$ maps with $n$ edges, $f$ faces and a marked discovery. A marked leaf can be retracted into a marked corner (see Figure \ref{feuille} left), such that there are $(2j+1)Q(j,f_2)$ maps with $j+1$ edges, $f_2$ faces and a marked leaf. There are $v_1Q(i,f_2)$ maps with $i$ edges, $f_1$ faces and a marked vertex. This explains why, in Theorems \ref{cs} and \ref{re}, Formulas \eqref{cut-slide} and \eqref{remy} are indeed consequences of the bijections. In a precubic map, one can retract a leaf into a marked side-edge losing two edges: after retracting the leaf into a corner,  merge the two edges adjacent to it (see Figure \ref{feuille} right). The same operation can be performed on the root of a precubic map, thus a precubic map with $3n+2$ edges and no leaves is equivalent to a cubic map with $3n$ edges. This explains how to derive Corollary \ref{cor} from \eqref{precubic}.

\begin{figure}[!h]
\captionsetup{width=0.8\textwidth}
\centering
\includegraphics[scale=0.5]{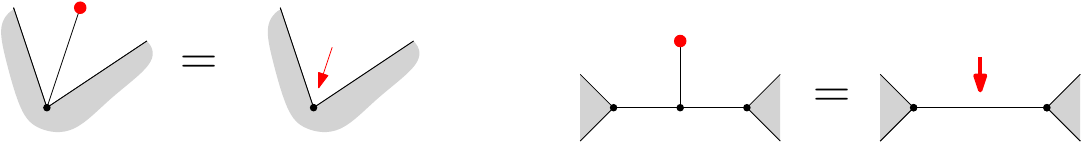}
\caption{Retracting a leaf: in a general map (left), in a precubic map (right, the vertex adjacent to the leaf is also removed)}
\label{feuille} 
\end{figure}
\subsection{Cut-and-slide bijection}

We will describe the bijection of Theorem \ref{cs} between maps $M$ with a marked discovery  and pairs of planar maps $(M_1,M_2)$ such that $M_1$ has a marked vertex and $M_2$ has a marked leaf. It will then be straightforward to see that restricting the bijection to precubic maps gives \eqref{precubic}.

We first need to introduce the notion of disconnecting discovery, and the splitting operation.
\begin{definition}
A discovery is said to be \textbf{disconnecting} if the corner preceding the discovery and the last corner (in the order defined by the exploration) around the discovery vertex lie in the same face. 
 
\label{disco}
\end{definition}
\begin{spli}
Any map with a marked disconnecting discovery can be split into two maps, one with a marked vertex, the other with a marked leaf in the outer face (see Figure \ref{split}): split the discovery vertex in two between $c$ and $c^*$. This divides the map in two, one containing the original root and a marked vertex, the other one rooted on the "splitting corner". Detach the discovery from the root to obtain a marked leaf in the outer face.
This operation is bijective: to go back, reattach the marked leaf to the root of the second map, and then glue its root to the last corner around the marked vertex in the first map.
\end{spli}
\begin{figure}[!h]
\captionsetup{width=0.8\textwidth}
\centering
\includegraphics[scale=1.5]{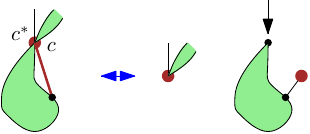}
\caption{Splitting a map at a disconnecting discovery. Here we only see what happens locally around the disconnecting discovery. On the left, the discovery and its discovery vertex are in fat brown, $c$ is the corner preceding the discovery, and $c^*$ is the last corner around the discovery vertex}
\label{split}
\end{figure}

The splitting operation of a disconnecting discovery describes our bijection in the case where the marked discovery is disconnecting, and the reverse bijection in the case where the marked leaf lies in the outer face. 

In general, discoveries are not disconnecting. In order to still split the map in two, given a discovery $e$, we will need to find a disconnecting discovery $e'$ that is "canonically" related to $e$. Conversely, the marked leaf in $M_2$ is not always in the outer face, we will need to "propagate" the leaf all the way up to the outer face. The notion of previous discovery will help us with that. This will be the general construction of the bijection that is detailed below.  First we need to ensure that there will always be a disconnecting discovery. That is the purpose of Lemma \ref{lemma}.

\begin{lem}
If a discovery leaves the outer face (as defined in Definition \ref{explo}), then it is disconnecting.
\label{lemma}
\end{lem}

\begin{proof}
We will show the following stronger result, which directly implies Lemma \ref{lemma}: if a vertex has a corner in the outer face, then its last corner lies in the outer face. Indeed, if a discovery leaves the outer face, then the corner preceding the discovery must be on the outer face. In particular, if $v$ denotes the discovery vertex, then $v$ has a corner on the outer face, so by this result the last corner of $v$ must be on the outer face. Thus, $v$ is a disconnecting discovery.

In a map $M$, let $v$ be a vertex with one of its corners lying in the outer face. If $v$ is the root vertex, then it is obvious that its last corner lies in the outer face. Else, let $e$ be the first edge around $v$, as defined by the exploration (i.e. the edge we see in the exploration just before we see $v$ for the first time). Let $c$ be a corner around $v$ that lies in the outer face. Suppose it is not the last corner around $v$ (if it is, we have nothing to prove). Let $c_0$ (resp. $c^*$) be the first (resp. the last) corner around $v$, and let $F_0$ (resp. $F^*$) be the face in which $c_0$ (resp. $c^*$) lies (see Figure \ref{corners}).

\begin{figure}[!h]
\captionsetup{width=0.8\textwidth}
\centering
\includegraphics[scale=1]{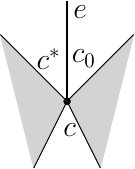}
\caption{}
\label{corners}
\end{figure}

If a face $F'$ is deeper (in terms of the partial order defined by the exploration) than a face $F$ (and $F\neq F'$), then during the exploration (ignoring all other faces), we first see a part of $F$, then all of $F'$, then the rest of $F$. This implies that, if $F^*$ is not the root face, during the exploration of $M$, at the time we see $c$, $F^*$ has not been discovered yet. But that would mean that at the time we see $e$ (which comes before $c$ in the exploration), $F^*$ has not been discovered yet. But then $e$ would be a discovery, and thus $c^*$ would be seen before $c_0$, that is a contradiction. The lemma is proved.
\end{proof}

%\begin{proof}[Proof of \cref{cs}]
\begin{bij}
\label{CS1}
The general process is iterative (see Figure \ref{bij} for an example).

\begin{figure}[!h]
\captionsetup{width=0.8\textwidth}
\centering
\includegraphics[scale=0.7]{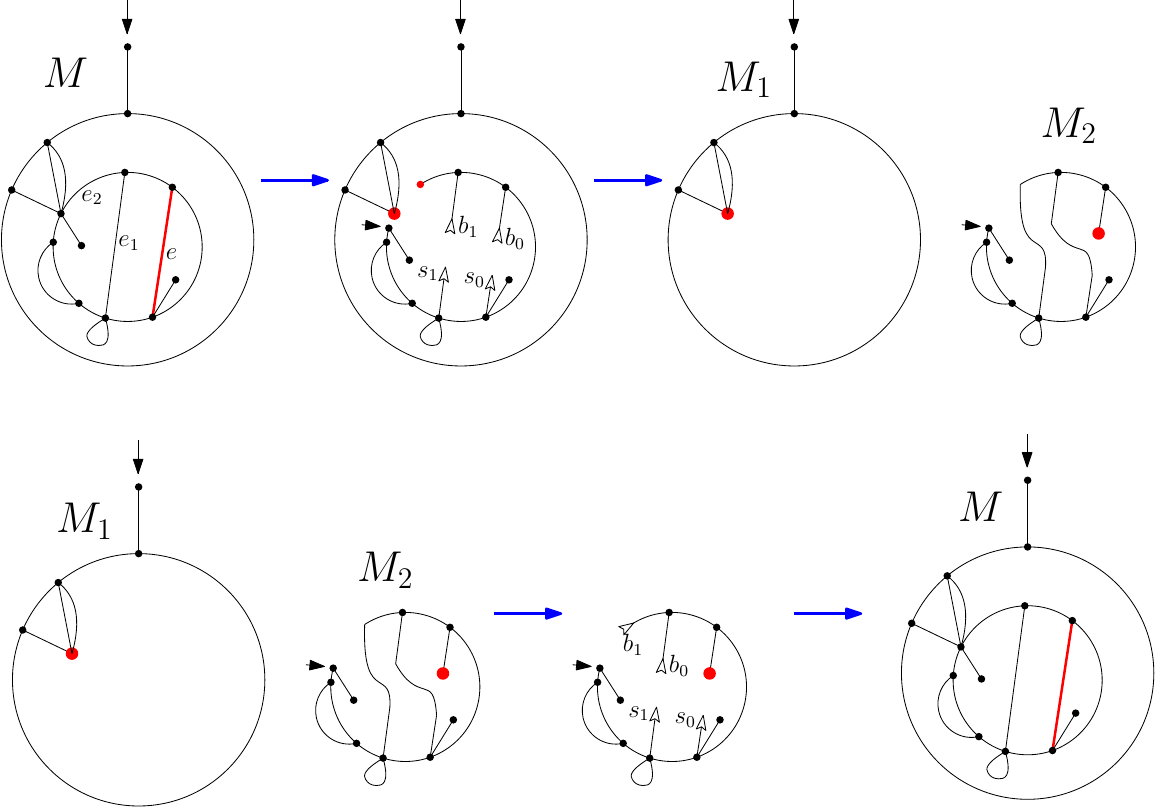}
\caption{The bijection (above) and its inverse (below)}
\label{bij}
\end{figure}

\textbf{Cut process:} Start from a map $M$ with a marked discovery $e$, let $v$ be its discovery vertex. If the discovery is disconnecting, then split $M$ at $v$ as described in the splitting operation.

Otherwise, open $e$ into a bud $b_0$ and a stem $s_0$, and consider its previous discovery $e_1$ (in the order defined above). If it is disconnecting, then split it, otherwise open it (into $b_1$ and $s_1$) and consider the previous discovery $e_2$, and so on until a splitting operation is made. Note that a discovery that leaves the outer face is always disconnecting (because of Lemma \ref{lemma}), so the algorithm terminates. One ends up with two maps $M_1$ and $M_2'$, such that $M_1$ has a marked vertex and $M_2'$ has a marked leaf $l$ and (possibly) some buds and stems, all lying in the outer face.

\textbf{Slide process:} We will not modify $M_1$. If there are no buds and stems in $M_2'$, we are done. Else, consider $s_0$, and make it a marked leaf. Then consider $l$, and make it a stem. Finally, glue back the buds and stems together canonically: starting from the root of $M_2'$, taking a clockwise tour of the outer face, one encounters a certain number of buds, then the same number of stems. There is only one way to match each bud with each stem such that the map remains planar. Equivalently, if there are $k+1$ buds and $k+1$ stems, match $b_0$ with $s_1$, and so on, until $b_k$ is matched with $l$. We obtain a map $M_2$ with a marked leaf, together with the map $M_1$ with a marked vertex.
\end{bij}
We can now describe the inverse bijection:
\begin{bij}
\label{CS2}
Starting from $M_2$ with a marked leaf $l$ and $M_1$ with a marked vertex, consider $M_2$. $l$ lies in a certain face $F$, and if $F$ is not the outer face, there is a certain discovery $e_0$ that enters $F$. Open it into a bud $b_0$ and a stem $s_0$, then open the previous discovery $e_1$, and repeat the process until a discovery that leaves the outer face has been opened (in that case there is no previous discovery to open). One ends up with a map $M_2'$ with a marked leaf $l$ and possibly some buds and stems, all lying in the outer face. If there are some buds and stems, let $s$ be the stem that was created last in the process. Make $s$ a marked leaf $l^*$, and make $l$ a marked stem $s^*$, then close the map canonically. One now has a map $M_2^*$ with a marked leaf on the outer face and (possibly) a marked edge $e$ (that comes from the closure of $s^*$). This marked edge is actually a discovery in $M_2^*$ (and will be a discovery in the final map). If $e$ does not exist, let $l^*=l$, and mark the edge adjacent to $l^*$ (and call it $e$). Now do the inverse of the splitting operation: glue $l^*$ to the root vertex of $M_2^*$ at its first corner, and then glue the root of the resulting map at the last corner of the marked vertex of $M_1$ to obtain a map $M$ with a marked discovery $e$.
\end{bij}
\begin{rem}\label{rem_precubic}
This operation restricts to precubic maps, and in this case discovery vertices are always of degree $3$, so that when split, the marked vertex of $M_1$ and the root of $M_2$ are both of degree $1$. This justifies \eqref{precubic}.
\end{rem}

%\end{proof}
Those two operations are inverse of each other because of the following property (that will be proved in Section \ref{sec:proof}):
\begin{lem}\label{lem_inverse}
In the bijection and its inverse, the closure of the buds and stems are discoveries in the resulting map. Moreover, in the bijection, if $e_1$,…$e_k$ are the discoveries created by closing buds and stems, $e_k$ leaves the outer face, and for all $i<k$, $e_{i+1}$ is the previous discovery of $e_i$. In the inverse bijection, if $e_1$,…$e_{k+1}$ are the discoveries created by closing buds and stems, $e_{k+1}$ is disconnecting, and for all $i<k$, $e_{i+1}$ is the previous discovery of $e_i$.
\end{lem}
\begin{rem}
\label{rem}
A similar but different cut-and-slide operation also appeared in \cite{Bettinelli} (and in some sense also in \cite{slice}). However there are significant differences: in \cite{Bettinelli} the cut path is geodesic (leftmost BFS), and both endpoints of the path need to be specified. Whereas here, the cut path is defined by a leftmost DFS, and only one endpoint of the path (the marked discovery) needs to be specified, the other endpoint (the disconnecting discovery) is uniquely determined. Furthermore, the bijections in \cite{Bettinelli} imply linear formulas, contrarily to our quadratic formulas.
\end{rem}
\subsection{Generalized R\'emy bijection}
\label{sec:Remy}
We will now describe a generalized R\'emy bijection on planar maps, that also relies on the cut-and-slide operation. 

We recall R\'emy's bijection for plane trees that proves the formula $(n+1)Cat(n)=2(2n-1)Cat(n-1)$ where $Cat(n)$ counts the number of rooted plane trees with $n$ edges (see Figure \ref{tree} for an example). Start from a tree $T$ with $n$ edges and a marked vertex $v$. If $v$ is a leaf, retract it (as in Figure \ref{feuille}) to obtain a tree $T'$ with $n-1$ edges, with a marked corner. Otherwise, $v$ has a last child $v'$, that is its child that is found last during a clockwise tour of the unique face. We can then contract the edge between $v$ and $v'$, and mark a corner around the merging vertex to remember where to grow the edge back (as in Figure \ref{tree}). The reverse bijection is then straightforward.

\begin{rem}
Leaves are defined as vertices of degree $1$ \textbf{that are not the root vertex}. Thus, if the root vertex has degree $1$, one should apply the second operation to it.
\end{rem}
\begin{figure}[!h]
\captionsetup{width=0.8\textwidth}

\centering
\includegraphics[scale=0.5]{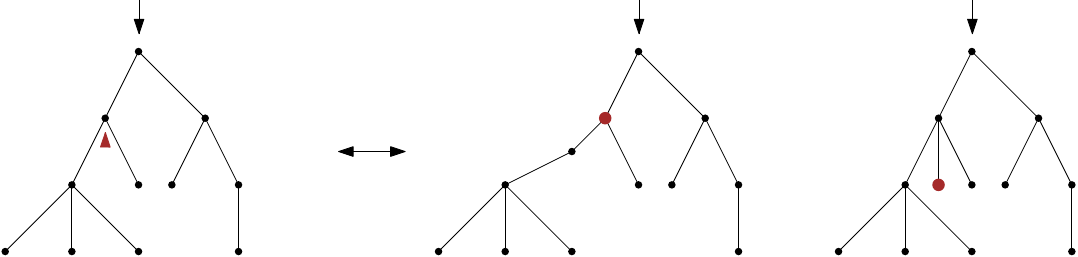}
\caption{R\'emy's bijection for plane trees. The left tree can be obtained by applying the bijection to each of the two trees on the right. The rightmost tree corresponds to the case where $v$ is a leaf, and the middle tree corresponds to the case where $v$ is not a leaf.}
\label{tree}
\end{figure}

In a general planar map, let us define precisely the operations of edge contraction and edge growing:

Let $v$ be a non-leaf vertex. It has a last corner $c_2$ (as defined by the exploration). Let $e$ be the edge that comes just before $c_2$ in clockwise order around $v$. We call $e$ the last edge around $v$. Let $v'$ be the other end of $e$ (note that it is possible that $v=v'$). It is called the last child of $v$. We will define the contraction operation only in the case where $v$ is seen strictly before $v'$ in the exploration. The edge $e$ is adjacent to two corners of $v$: $c_2$ and another corner $c_1$. It is also adjacent to two corners of $v'$, $c'$ and $c''$. Since $v$ appears strictly before $v'$, $c'$ is the first corner around $v'$, and $c''$ is the last corner around $v'$.

The contraction operation consists in contracting $e$, merging $v$ and $v'$ into a vertex $v^*$, merging $c_1$ and $c'$ into $c^*$ and $c_2$ and $c''$ into $c$, such that $c$ is the last corner around $v^*$, and $c^*$ is the marked corner.

Conversely, starting with a marked corner $c^*$ adjacent to a vertex $v^*$ whose last corner is $c$, one can grow an edge $e$, splitting $v^*$ into $v$ and $v'$, $c^*$ into $c_1$ and $c'$, and $c$ into $c^*$ and $c_2$. $c_2$ is the last corner around $v$ and $c''$ the last corner around $v'$. The marked vertex is $v$, and $v'$ is its last child. This defines the growing operation. 
Lemma \ref{lem_growing} will ensure these two operations are inverse of each other.

\begin{figure}[!h]
\captionsetup{width=0.8\textwidth}
\centering
\includegraphics[scale=0.7]{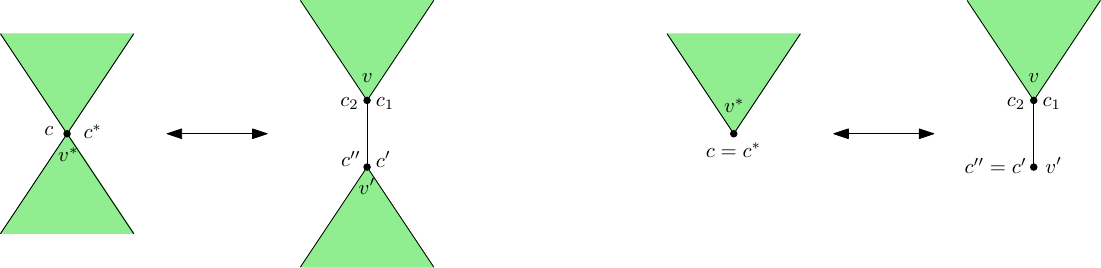}
\caption{The contraction and growing operations (right: the special case where the marked corner is the last corner)}
\label{growing}
\end{figure}

\begin{lem}\label{lem_growing}
Let $M$ be a map with a marked corner $c^*$, and let $M'$ be the map obtained after a growing operation is performed at $c^*$. Let $v$ be the marked vertex and $v'$ its last child. Then $v$ appears before $v'$ in the exploration.
\end{lem}
\begin{proof}
The growing operation is local, so both explorations of $M$ and $M'$ look exactly the same before $c^*$ (resp. $c_1$) is found. Let $p$ be the first corner around $v^*$ in $M$. After the growing operation it is adjacent to $v$, and it is reached before any other corner around $v$ or $v'$, thus $v$ appears before $v'$ in the exploration (the special case $p=c^*$ works the same).
\end{proof}
We also need to cover the remaining case, namely when the the marked vertex is seen after its last child in the exploration.
\begin{lem}
Let $v$ be a non-leaf vertex, $e$ its last edge, and $v'$ its last child. If $v'$ is found before $v$ in the exploration (including the case $v=v'$), then $e$ is a discovery.
\end{lem}
\begin{proof}
If $v=v'$, then $e$ is a loop, thus a discovery (since a loop separates a planar map in two, it has to be opened during the exploration).

Otherwise, let $c_1,c_2$ (resp. $c',c''$) be the corners of $v$ (resp. $v'$) that are adjacent to $e$. Note that $c_2$ is the last corner of $v$, and assume $e$ is not a discovery. Thus, it is not opened during the exploration, such that $c''$ comes before $c_2$, and $c_1$ before $c'$ (see Figure \ref{growing_not}). $c'$ cannot be the first corner of $v'$, otherwise $v$ would be seen before $v'$ in the exploration. Thus, Lemma \ref{clockwise} implies that $c'$ comes after $c''$. But that would mean the order among the four corners in the exploration is $c''$, then $c_2$, then $c_1$, then $c'$, which is impossible since $c_2$ is the last corner of $v$. Thus, $e$ is a discovery.
\end{proof}
\begin{figure}[!h]
\captionsetup{width=0.8\textwidth}
\centering
\includegraphics[scale=0.7]{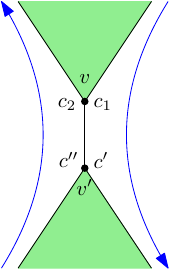}
\caption{The exploration locally around $e$, if it is not a discovery}
\label{growing_not}
\end{figure}
Now we are ready to explain the generalized R\'emy bijection:

\begin{bij}
\label{RM}
Take a planar map $M$ and mark a vertex $v$. If it is a leaf, contract it to mark a corner. Else, let $v'$ be its last child, and $e$ be its last edge. If $v'$ is seen after $v$ in the exploration, then contract $e$ and mark the corresponding corner. The inverse operation is the growing operation as described above. This gives the first term of the RHS of \eqref{remy}.
Otherwise, $e$ is a discovery, apply the cut-and-slide operation at $e$. One ends up with two maps $(M_1,M_2)$. $M_1$ has a marked vertex, and $M_2$ has a marked leaf $l$ that is the last child of its neighbor $w$. Contract $l$ to mark $w$. To go backwards, grow $l$ out of the last corner around $w$, then apply the inverse of the cut-and-slide operation. This gives the second term of the RHS of \eqref{remy}.

% consider the last corner around $v$, and let $e$ be the edge that appears just before the last corner around $v$ in the clockwise order. Let $v'$ be the other end of $e$ (note that it is possible that $v=v'$). $v'$ will play the role of the "last son" of $v$. We can try and contract $e$, marking a corner around the merging vertex to remember where to grow back the edge. Consider the inverse of this operation (growing an edge from $v$) : it produces a vertex $v'$ that has to be found after $v$ in the exploration. Thus, the operation of contracting $e$ is well-defined whenever $v$ is found before $v'$ in the exploration. This gives us the first term of the RHS of  \cref{remy}. Theremaining case is when $v'$ has actually been seen before $v$ during the exploration (including the case $v=v'$). But this means we reached the end of a cycle, and since we are dealing with planar maps, it means that $e$, the edge we were trying to contract, is actually a discovery ! In that case, instead of contracting it, we can just apply the cut and slide operations (as defined in the previous bijection) with $e$ as a marked discovery. We end up with two maps $M_1$ and $M_2$, but now the marked leaf of $M_2$ has to be the last son of its neighbor $v^*$. So we can just contract it and mark $v^*$ (to go back, we just grow the leaf from the last corner of $v^*$). This gives us the second term of the RHS of \cref{remy}.
\begin{figure}[!h]
\captionsetup{width=0.8\textwidth}
\centering
\includegraphics[scale=0.5]{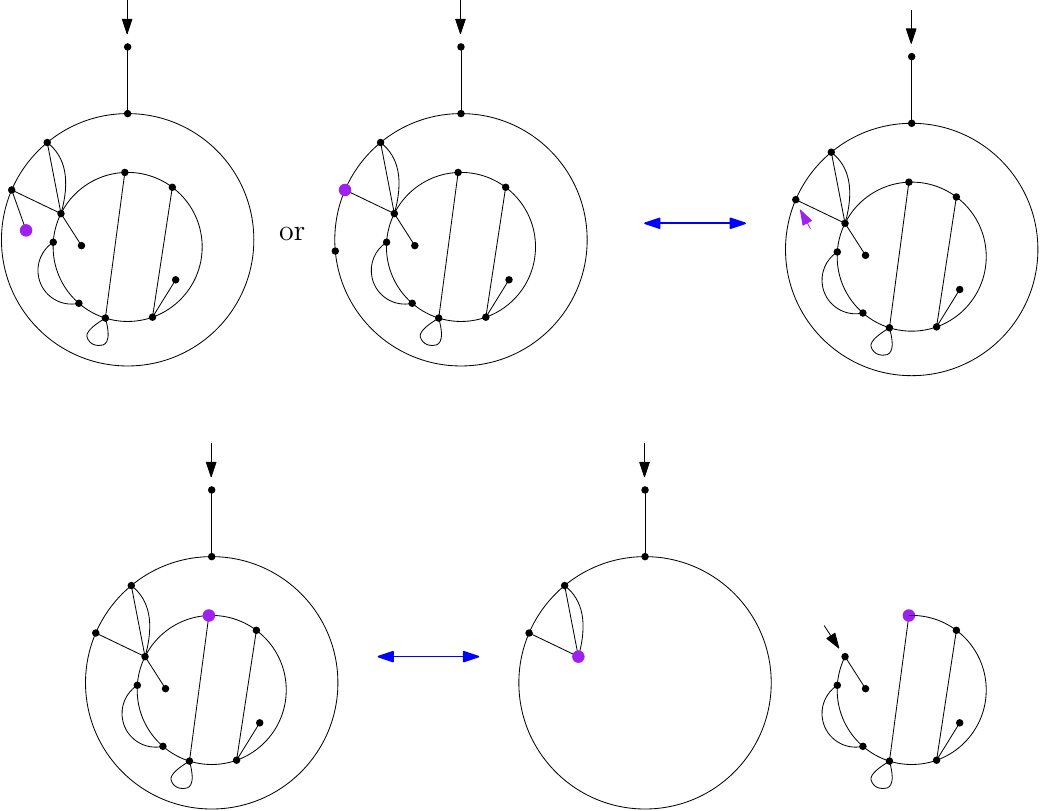}
\caption{The generalized R\'emy bijection. The case similar to trees (above), and the cut-and-slide case (below)}
\label{remy2}
\end{figure}
\end{bij}

%\begin{rem}
%There is also a generalization of R\'emy's bijection on binary trees to precubic maps, that works exactly the same, we leave it as an exercise to the reader.
%\end{rem}

\section{Proofs}

\label{sec:proof}

%In order to prove that the cut-and-slide operation (and its inverse) presented in \cref{sec:bij} is really a bijection, we will describe what happens to the different faces of the map during the operation (see \cref{faces}). 

In this section we will present a proof that Bijections \ref{CS1} and \ref{CS2} are inverse to each other, thus implying that the cut-and-slide operation is a bijection. Then, together with the lemmas presented in Section \ref{sec:Remy}, it will be enough to prove that the generalized Rémy bijection (Bijection \ref{RM}) is indeed a bijection. To do so, we will describe what happens to the different faces of the map during the operation (see Figure \ref{faces}).

\begin{proof}[Proof of Lemma \ref{lem_inverse}]

We only prove Lemma \ref{lem_inverse} for Bijection \ref{CS1}. The proof for Bijection \ref{CS2} is very similar and can be copied almost verbatim.

Thanks to the exploration, we can consider faces as words on the alphabet of side-edges (each face is described by the list of the side-edges that lie inside it, ordered as they appear during the exploration). For any side-edge $e$, we will call $\overline{e}$ its opposite side-edge (i.e the other side of this edge). By abuse of notation we may also refer to the edge $e$ that is the edge made of the side-edges $e$ and $\overline{e}$, but we will stress on the fact that we are referring to an edge.

Fix a planar map $M$. Let $d_0$ be a discovery, $d_1$ its previous discovery (if it exists), and so on, $d_{i+1}$ being the previous discovery of $d_i$, as long as it exists. For all $i$, the discovery $d_i$ discovers the face $F_i$ while leaving face $F_{i+1}$. For all $i$, let $v_i$ be the discovery vertex corresponding to $d_i$, and let us tell apart both side-edges of the discovery: $d_i$ sits in $F_{i+1}$ and $\overline{d_i}$ sits in $F_i$. Let $k$ be the smallest $i$ such that $d_i$ is a disconnecting discovery. For all $i$ between $1$ and $k$, we can set $F_i=\overline{d_i}A_id_{i-1}B_i$, that is $A_i$ are the side-edges of $F_i$ encountered before $d_{i-1}$ in the exploration, and $B_i$ are the the side-edges of $F_i$ encountered after $d_{i-1}$ in the exploration. We also set $F_0=\overline{d_0}A_0$ and $F_{k+1}=\overline{d_{k+1}}Cd_kB_{k+1}C'$ where $B_{k+1}$ is the word of side-edges of $F_{k+1}$ that appear between $d_k$ and the last corner around $v_k$.\\
Now let us describe the cut-and-slide operation when $d_0$ is marked. First, for all $i<k$, all \textbf{edges} $d_i$ are opened into a bud $b_i$ and an stem $s_i$. The \textbf{edge} $d_k$ becomes a marked leaf $s_k$. Then $s_k$ becomes a stem, and $s_0$ becomes the marked leaf. Then $b_i$ is matched with $s_{i+1}$ to create an edge $e_i$. The map has been split into two distinct maps $M_1$ and $M_2$, now let us see what happened to the faces. In $M_1$, there is only one face that has been modified, it is the face in which lies the last corner of the marked vertex $v^*$, that we call $G^*$. It is straightforward to see that $G^*=\overline{d_{k+1}}CC'$ where the last corner of $v^*$ appears exactly between the end of $C$ and the beginning of $C'$. In $M_2$, the modified faces will be called $G_i$. We have $G_0=\overline{e_0}A_0s_0\overline{s_0}B_1$, $G_i=\overline{e_i}A_ie_{i-1}B_{i+1}$ and $G_k=A_ke_{k-1}B_{k+1}$. $G_k$ is the outer face of the second map, and for all $i$, $G_{i+1}$ is the previous face of $G_i$.\\
\begin{figure}[!h]
\captionsetup{width=0.8\textwidth}
\centering
\includegraphics[scale=1]{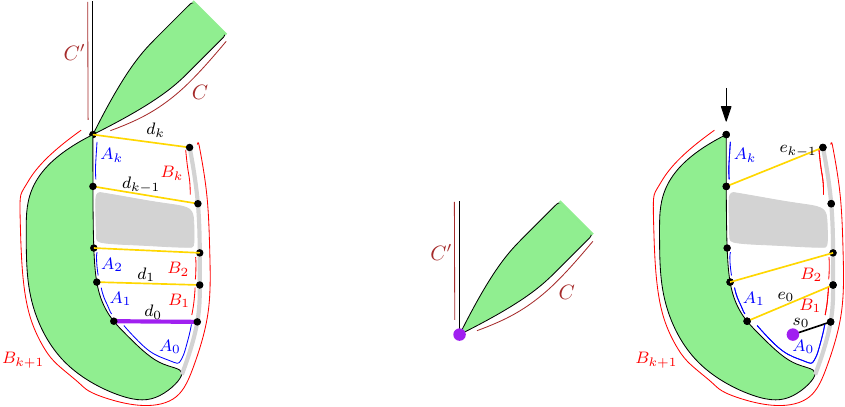}
\caption{What happens to the faces during the cut-and-slide operation. Marked objects are in fat purple, other discoveries involved are in gold. Blue contour materializes the $A_i$'s and red contour the $B_i$'s}
\label{faces}
\end{figure}\\
We must now prove that, for all $i<k$, the \textbf{edge} $e_i$ is a discovery that leaves $G_{i+1}$ and discovers $G_i$. We already know that $e_i$ lies in $G_{i+1}$ and $\overline{e_i}$ lies in $G_i$, we just have to prove that the \textbf{edge} $e_i$ is a discovery, i.e $e_i$ is the first side-edge $e$ of $G_{i+1}$ such that $\overline{e}$ lies in $G_i$. Assume the contrary. There must be a side-edge $f$ in $A_{i+1}$ such that $\overline{f}$ lies in $G_i$. There are two possibilities:
\begin{enumerate}
\item Either $\overline{f}$ belongs to $A_i$, but then in the original map, in $F_{i+1}$, $f$ comes before $d_i$ and $\overline{f}$ lies in $F_i$, thus $d_i$ could not be the discovery, in this case there would be a contradiction. 
\item Or $\overline{f}$ belongs to $B_{i+1}$, but then in the original map, in $F_{i+1}$ we see the side-edges $f$, then $d_i$, then $\overline{f}$ in that order. But since we deal with planar maps, an edge which has both sides in the same face is a cut-edge (removing it disconnects the map). But then in that configuration, $d_i$ must be a disconnecting discovery (see Figure \ref{moche}). Indeed, let $M^*$ be the connected component containing $d_i$ that we obtain after removing $e$. Then, in the map $M^*$, $d_i$ would be a discovery that leaves the outer face, thus disconnecting (because of Lemma \ref{lemma}). This is a contradiction, because $i<k$.
\end{enumerate}  We proved what we wanted. 

\end{proof}

\begin{figure}[!h]
\captionsetup{width=0.8\textwidth}
\centering
\includegraphics[scale=0.5]{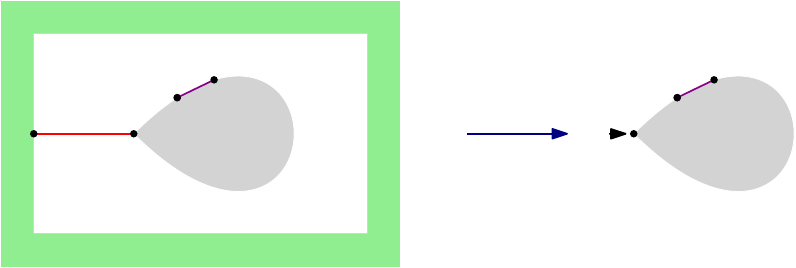}
\caption{The impossibility of a cut-edge separating a non-disconnecting discovery from the root. In red, the cut-edge, in brown the discovery}
\label{moche}
\end{figure}

Knowing exactly what happens to the faces and that the discoveries are in some sense "preserved" is enough to prove that both the cut-and-slide operation and its inverse work fine and are actually inverse operations of each other.

\section{Controling the degrees of the vertices}
\label{sec:deg}
Here we present an analogue of \eqref{cut-slide} with control over the degrees of the vertices. This is a result of independent interest that is not used in the other proofs of the paper.

If $\textbf{v}=(v_i)_{i\in \mathbb{N}}$ is a sequence of integers, for any $j>0$, we set $\delta_j(\textbf{v})=\textbf{w}$ where $w_j=v_j+1$ and $w_i =v_i$ for $i\neq j$, and $\delta_{-j}(\textbf{v})=\textbf{w'}$ where $w_j'=v_j-1$ and $w_i '=v_i$ for $i\neq j$. Finally, we set $\delta(\textbf{v},j_1,\hdots,j_k)=\delta_{j_1}\circ \hdots\circ \delta_{j_k}(\textbf{v})$.

Let $M(r,f,\textbf{v})$ be the number of planar maps with $f$ faces, with root of degree $r$, with $\textbf{v}=(v_i)_{i\in \mathbb{N}}$ such that there are $v_i$ vertices of degree $i$ (root included). The cut-and-slide operation only modifies the degrees at the marked leaf and the splitting vertex, so we can immediately derive this more precise formula:
\begin{thm}\label{thm_degrees}
\begin{align*}
(f-1)M(r,f,\textbf{v})=&\sum_{\substack{j,k\geq 1}}\sum_{\substack{\textbf{u},\textbf{w}}}\sum_{\substack{f_1+f_2=f\\f_1,f_2\geq 1}}(u_j-\mathbb{1}_{j=r})M(r,f_1,\textbf{u})(w_1-\mathbb{1}_{k=1})M(k,f_2,\textbf{w})\\
&+\sum_{\substack{j+k=r-1}}\sum_{\textbf{u}+\textbf{w}=\delta(\textbf{v},1,-r,k,j)}\sum_{\substack{f_1+f_2=f\\f_1,f_2\geq 1}}M(j,f_1,\textbf{u})(w_1-\mathbb{1}_{k=1})M(k,f_2,\textbf{w}).
\end{align*}
where the sum of the first line is over pairs of sequences of numbers $(\textbf{u},\textbf{w})$ $\textbf{u}+\textbf{w}=\\ \delta(\textbf{v},1,j,k,-(j+k+1))$.
\end{thm}

\begin{proof}
Notice that, in the cut-and-slide operation, very few of the vertex degrees are modified: there is the leaf that is created in $M_2$, and the disconnecting discovery vertex that is split in three. Other than that, for all other vertices, although some of their adjacent edges might be split, they are reattached somewhere else so their degree does not change.

Let $v$ be the vertex that is split in three: it gives birth to $v_1$, the marked vertex in $M_1$, $v_2$, the root vertex of $M_2$, and $v_3$, that after some possible transformation "becomes" the marked leaf in $M_2$. Say $v$ is of degree $j+k+1$, with $deg(v_1)=j$ and $deg(v_2)=k$.

The $(w_1-\mathbb{1}_{k=1})$ term means that the marked leaf in $M_2$ cannot be the root. Finally, there are two possible cases, depending on whether $v$ is the root of $M$ or not, each implying one term in the RHS.
\end{proof}

\begin{rem}
$\delta_j(\textbf{v})$ means there is one more vertex of degree $j$, and $\delta_{-j}(\textbf{v})$ means there is one less vertex of degree $j$. This notation is somehow complicated but avoids dealing with special cases, for instance in $\delta(\textbf{v},j,k)$ there is no problem if $j=k$, contrarily to saying something like $v_j'=v_j+1$, $v'_k=v_k+1$ and $v_i'=v_i$ for $i\neq j,k$.
\end{rem}
 Note that this recurrence formula allows us to compute the number of maps with bounded vertex degrees, i.e. it can be specialized to maps with vertex degrees $\leq d$ for some $d$. Restricting to vertex degrees $1$ and $3$, one recovers \eqref{precubic}.\\

Formula \eqref{remy} also has an analog where the degrees are recorded, but it involves more complicated cases and is less useful for enumeration, we leave it as an exercise to the reader.

\section{The proof of the Carrell-Chapuy formula in the planar case}
\label{sec:calcul}

We are now ready to prove the Carrell-Chapuy formula  in the planar case \eqref{planar}. We will prove this by induction on $n$, only knowing that $Q(0,f)=\mathbb{1}_{f=1}$.\\
Applying \eqref{cut-slide} to the dual map, we obtain the following formula, which will be helpful for the proof:
\begin{cor}

\begin{equation}
(v-1)Q(n,f)=\sum_{\substack{i+j=n-1\\i,j\geq 0}}\sum_{\substack{f_1+f_2=f+1\\f_1,f_2\geq 1}}(2i+1)Q(i,f_1)f_2Q(j,f_2)
\label{dual}
\end{equation}
\end{cor}

Starting with \eqref{cut-slide}, and applying \eqref{remy}, then in a second time applying \eqref{cut-slide} backwards, we obtain

\begin{align*}
(f-1)Q(n,f)=&(2n-1)Q(n-1,f-1)+2(2n-1)Q(n-1,f)\\&+2\sum_{\substack{i+j=n-2\\i,j\geq 0}}\sum_{\substack{f_1+f_2=f\\f_1,f_2\geq 1}}(2i+1)Q(i,f_1)(2j+1)Q(j,f_2)\\&+\sum_{\substack{i+j+k=n-2\\i,j,k\geq 0}}\sum_{\substack{f_1+f_2+f_3=f\\f_1,f_2,f_3\geq 1}}(2i+1)Q(i,f_1)v_2Q(j,f_2)v_3Q(k,f_3)
\\&=(2n-1)Q(n-1,f-1)+2(2n-1)Q(n-1,f)\\&+2\sum_{\substack{i+j=n-2\\i,j\geq 0}}\sum_{\substack{f_1+f_2=f\\f_1,f_2\geq 1}}(2i+1)Q(i,f_1)(2j+1)Q(j,f_2)\\&+\sum_{\substack{i+j=n-1\\i,j\geq 0}}\sum_{\substack{f_1+f_2=f\\f_1,f_2\geq 1}}(f_1-1)Q(i,f_1)v_2Q(j,f_2).
\end{align*}
So adding \eqref{remy} to this
\begin{align*}
(n+1)Q(n,f)=&(2n-1)Q(n-1,f-1)+2(2n-1)Q(n-1,f)\\&+2\sum_{\substack{i+j=n-2\\i,j\geq 0}}\sum_{\substack{f_1+f_2=f\\f_1,f_2\geq 1}}(2i+1)Q(i,f_1)(2j+1)Q(j,f_2)\\&+\sum_{\substack{i+j=n-1\\i,j\geq 0}}\sum_{\substack{f_1+f_2=f\\f_1,f_2\geq 1}}(i+1)Q(i,f_1)v_2Q(j,f_2).
\end{align*}
Let \[S=\sum_{\substack{i+j=n-1\\i,j\geq 0}}\sum_{\substack{f_1+f_2=f\\f_1,f_2\geq 1}}(i+1)Q(i,f_1)v_2Q(j,f_2).\]
We want to prove 
\[S=(2n-1)Q(n-1,f-1)+\sum_{\substack{i+j=n-2\\i,j\geq 0}}\sum_{\substack{f_1+f_2=f\\f_1,f_2\geq 1}}(2i+1)Q(i,f_1)(2j+1)Q(j,f_2).\]
 We apply the recursion hypothesis:
\begin{align*}
S&=\sum_{\substack{i+j=n-1\\i,j\geq 0}}\sum_{\substack{f_1+f_2=f\\f_1,f_2\geq 1}}(2(2i-1)Q(i-1,f_1)+2(2i-1)Q(i-1,f_1-1))v_2Q(j,f_2)\\&+3\sum_{\substack{i+j+k=n-3\\k,l\geq 0}}\sum_{\substack{f_1+f_2+f_3=f\\f_1,f_2,f_3\geq 1}}(2i+1)Q(i,f_1)(2j+1)Q(j,f_2)v_3Q(k,f_3)\\&+vQ(n-1,f-1).
\end{align*}
But, according to \eqref{cut-slide}, \[\sum_{\substack{i+j=n-1\\i,j\geq 0}}\sum_{\substack{f_1+f_2=f\\f_1,f_2\geq 1}}(2(2i-1)Q(i-1,f_1-1))v_2Q(j,f_2)=2(f-2)Q(n-1,f-1),\]
and 
\begin{align*}
&\sum_{\substack{i+j+k=n-3\\k,l\geq 0}}\sum_{\substack{f_1+f_2+f_3=f\\f_1,f_2,f_3\geq 1}}(2i+1)Q(i,f_1)(2j+1)Q(j,f_2)v_3Q(k,f_3)\\&=\sum_{\substack{i+j=n-2\\i,j\geq 0}}\sum_{\substack{f_1+f_2=f\\ f_1,f_2\geq 1}}(2i+1)Q(i,f_1)(f_2-1)Q(j,f_2).
\end{align*}
So
\begin{align*}
S=&(2(f-2)+v)Q(n-1,f-1)+\sum_{\substack{i+j=n-2}}\sum_{\substack{f_1+f_2=f\\f_1,f_2\geq 1}}(2i+1)Q(i,f_1)(2j+1)Q(j,f_2)\\&+\sum_{\substack{i+j=n-2}}\sum_{\substack{f_1+f_2=f\\f_1,f_2\geq 1}}(2i+1)Q(i,f_1)f_2Q(j,f_2).
\end{align*}
But using \eqref{dual}, we have 
\[\sum_{\substack{i+j=n-2}}\sum_{\substack{f_1+f_2=f\\f_1,f_2\geq 1}}(2i+1)Q(i,f_1)f_2Q(j,f_2)=(v-1)Q(n-1,f-1),\]
which finishes the proof.

\begin{rem}
The proof above is not straightforward, and since it uses duality, our method cannot be applied to finding an extended Carrell-Chapuy formula with control over the degrees (if it even exists). %It cannot be used either for proving the planar case of the analogue formula on bipartite maps found by Kazarian and Zograf \cite{KZ}.
\end{rem}

\section{A bijection for precubic maps with two faces}
\label{sec:twofaces}
In this section, we will give a sketch of the proof of a first step towards uniting higher genus and multiple faces: the case of two-faced precubic maps.

Similarly as \eqref{precubic}, we have a recurrence formula for higher genus precubic maps with two faces:
\begin{equation}
\label{deuxfaces}
(2g+1)\alpha_g(n,2)=\alpha^{(3)}_{g-1}(n,2)+\sum\limits_{g_1+g_2=g} \sum\limits_{i+j=n}\alpha^{(1)}_{g_1}(i,1)\alpha^{(1)}_{g_2}(j,1)
\end{equation}
where $\alpha_g(n,f)$ counts the number of precubic maps with $n$ edges and $f$ faces, of genus $g$, and $\alpha^{(k)}_g(n,f)$ counts the number of those maps with $k$ marked leaves.

We can define the exploration of a precubic map with two faces in the following way:
\begin{definition}
We will describe an exploration of the map as a canonical labeling of all its corners (see Fig. \ref{twofac}).

As in Definition \ref{explo}, we can define the discovery as the first edge adjacent to both faces that is encountered in a clockwise tour of the root face starting from the root, and the discovery vertex as the vertex that appears right before the discovery in this tour (there is only one discovery since there are 2 faces). Opening the discovery into a bud and a stem creates a blossoming (i.e. with a bud and a stem) unicellular map. In this map, it is possible to label all the corners in their order during the tour of the unique face. Thus, there is a labeling of all the corners of the map, and a discovery vertex (we will not need the discovery itself in what follows).
\end{definition}
The discovery vertex is obviously of degree $3$ (because a leaf is adjacent to only one face).  We can now introduce special vertices and \textit{trisections}.
\begin{definition}
A trisection is a vertex whose corner labels are in counterclockwise order around this vertex. A vertex is said to be \textit{special} if it is a trisection or if it is the discovery vertex.
\end{definition}
A trisection of the map is exactly a trisection of the unicellular (blossoming) map. In \cite{unicellular,trisections}, it is proven that in a unicellular map of genus $g$, there are $2g$ trisections. Furthermore, we can easily verify that the discovery vertex is not a trisection, thus the following lemma holds:
\begin{lem}
There are $2g+1$ special vertices in a two-faced precubic map of genus $g$ (see Fig. \ref{twofac} right).
\end{lem}

\begin{figure}[!h]

\captionsetup{width=0.8\textwidth}

\centering
\includegraphics[scale=1]{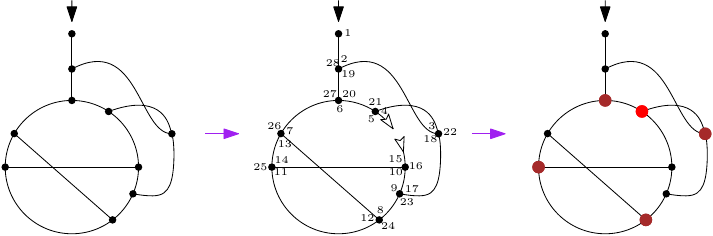}
\caption{A precubic map (left), its exploration (center), and its special points in fat (right, discovery vertex in red, trisections in brown)}
\label{twofac}
\end{figure}

The operation we will consider is fairly simple: take a map of genus $g$ with a marked special vertex, and split it.
There are two possible cases:\begin{enumerate}

\item the map is disconnected into two maps with a marked leaf each, such that the genii, number of edges and faces add up,
\item the resulting map is of genus $g-1$ and has $3$ marked leaves, along with information on how to glue back the leaves together (given $3$ leaves, there are $2$ possible ways of gluing them together).
\end{enumerate}

\begin{rem}
The first case can only appear if the special vertex is the discovery vertex. In the second case, the special vertex can be either the discovery vertex or a trisection.
\end{rem}
There is a bijection between maps in case (1) and pairs of maps with a marked leaf each, the inverse operation being the same as the case of disconnecting discoveries in the planar case: given two maps with a marked leaf each, it is possible to glue them back together as in Figure \ref{split}, and one obtains a map with a marked discovery vertex. 

Case (2) is more complicated: given a map of genus $g-1$ with $3$ marked leaves, that we will call a \textit{tripod}, there are two ways to glue back the leaves. A gluing is said to be \textit{valid} if the resulting map has genus $g$ and the gluing vertex is a special vertex. Tripods can have $0$, $1$ or $2$ valid gluings, and we will provide a classification of tripods with respect to their number of gluings. In the following, we will refer to "marked leaves" as just "leaves" as there is no risk of ambiguity. In order to prove \eqref{deuxfaces}, we will then have to prove there is a bijection between tripods with respectively $0$ and $2$ valid gluings.

\begin{definition}
In a given map, let $v$ be the discovery vertex. It has two corners in the root face, and is thus seen twice in the tour of the root face. We can decompose the root face $F$ as a word on side-edges as in Section \ref{sec:proof}, starting from the root, as $F=TvOv\overline{T}$ (see Fig. \ref{tripod} left).

Take a given tripod $M$, with exactly two leaves in the root face, both in $\overline{T}$. Glue those two leaves to split the root face in two. Let $F'$ be the face thus obtained that is not the root face. We say that $M$ is in the special case if no side-edge of $F'$ has its opposite side-edge in $T$  (see Fig. \ref{tripod} right).
\end{definition}
\begin{figure}[!h]
\captionsetup{width=0.8\textwidth}

\center
\includegraphics[scale=1]{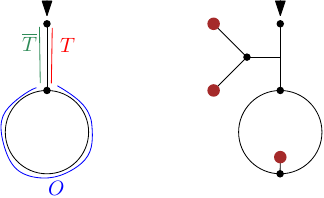}
\caption{The schematic decomposition of the root face (left), and the special case (right, with the leaves in fat brown)}
\label{tripod}
\end{figure}

\begin{lem} The following classification holds:
\begin{itemize}
\item A tripod with all leaves in the same face or in the special case or with two leaves in $O$ has $1$ valid gluing.
\item A tripod having exactly one leaf in the root face has $2$ valid gluings.
\item A tripod with exactly two leaves in the root face, at least one of which in $T$ or $\overline{T}$ (except for the special case), has no valid gluings.
\end{itemize}
\end{lem}
The proof of this lemma is done by carefully considering the cases implied (we omit it here).

To finish the proof, one needs to find a bijection between tripods with $0$ valid gluings and tripods with $2$ valid gluings. The bijection consists of cleverly "unplugging" the discovery and plugging it somewhere else to "transfer" leaves between the faces. The proof involves several cases and is a bit technical, thus we prefer to omit it. This shows that on average, a tripod has one valid gluing, implying \eqref{deuxfaces}.

\section*{Acknowledgements}
The author wishes to thank his advisor Guillaume Chapuy for suggesting the problem and for useful discussions, as well as the anonymous reviewers, whose comments allowed to make this article clearer.
\section*{Funding}
This work was supported by ERC-2016-STG 716083 "CombiTop".
\bibliography{LOUFPlanarMaps.bib}{}
\bibliographystyle{plain}

%\vspace{0.5cm}

%\vspace{1cm}

%\newpage
%%\vspace{0.1cm}
%%

%\tableofcontents

 \end{document}